\documentclass{amsart}
\usepackage{enumerate}
\usepackage{amssymb}
\usepackage{latexsym,amscd,amsfonts}
\def \Z {\mathbb Z}
\def \R {\mathbb R}

\def \N {\mathbb N}

\def \E {\mathbb E}
\def\dd{\mathrm d}
\def\ee{\mathrm e}
\def\ii{\mathrm i}
\newcommand{\SpN}{\mathrm{Sp}_{\mathrm{N}}(\R)\,}
\newcommand{\spN}{\mathfrak{sp}_{\mathrm{N}}(\R)\,}
\newcommand{\supp}{\mathrm{supp}\;}
\newcommand{\Ho}{H(\omega)}
\newcommand{\Hl}{H_{\ell}(\omega)}

\newcommand{\LpSI}{$p$-contracting and $L_p$-strongly irreducible, for every $p\in \{1,\ldots,N\}$}
\newcommand{\aN}{\{1,\ldots,N\}}
\newcommand{\omO}{\omega^{(0)}}
\newcommand{\SN}{\mathrm{S}_{\mathrm{N}}(\R)\,}
\newcommand{\SV}{\mathcal{S}_{\mathrm{V}}}
\DeclareMathOperator{\dlO}{d_{\log\, \mathcal{O}}}

\newtheorem{thm}{Theorem}
\newtheorem{prop}{Proposition}
\newtheorem{cor}{Corollary}
\newtheorem{defi}{Definition}
\newtheorem{lem}{Lemma}
\newtheorem{rem}{Remark}
\begin{document}
       
\title[Anderson with generic interaction potential]{Localization for an Anderson-Bernoulli model with generic interaction potential}
\author{Hakim Boumaza} 
\email{boumaza@math.univ-paris13.fr}
\address{UMR CNRS 7539 - D\'epartement de Math\'ematiques \\
Institut Galil\'ee \\ 
Universit\'e PARIS 13 \\ 
99 avenue Jean-Baptiste Cl\'ement \\ 
93430 VILLETANEUSE \\
FRANCE}
\thanks{}
\begin{abstract}
We present a result of localization for a matrix-valued Anderson-Bernoulli operator, acting on $L^2(\R)\otimes \R^N$, for an arbitrary $N\geq 1$, whose interaction potential is generic in the real symmetric matrices. For such a generic real symmetric matrix, we construct an explicit interval of energies on which we prove localization, in both spectral and dynamical senses, away from a finite set of critical energies. This construction is based upon the formalism of the F\"urstenberg group to which we apply a general criterion of density in semisimple Lie groups. The algebraic nature of the objects we are considering allows us to prove a generic result on the interaction potential and the finiteness of the set of critical energies.

\end{abstract}


\maketitle
\section{Introduction}\label{sec_intro}

In this article, we will discuss a generic result on localization properties for the following random family of quasi one-dimensional Anderson-Bernoulli operators : 
\begin{equation}\label{eq_model_Hl}
\Hl = -\frac{\dd^2}{\dd x^2}\otimes I_{\mathrm{N}} + V + \sum_{n\in \Z}  \left(
\begin{smallmatrix}
c_1 \omega_{1}^{(n)} \mathbf{1}_{[0,\ell]}(x-\ell n) & & 0\\ 
 & \ddots &  \\
0 & & c_N \omega_{N}^{(n)} \mathbf{1}_{[0,\ell]}(x-\ell n)\\ 
\end{smallmatrix}\right),
\end{equation}
\noindent acting on $L^2(\R)\otimes \R^N$, where $N\geq 1$ is an integer, $I_{\mathrm{N}}$ is the identity matrix of order $N$ and $\ell>0$ is a real number. The matrix $V$ is a real $\mathrm{N}\times \mathrm{N}$ symmetric matrix, the space of these matrices being denoted by $\SN$. The constants $c_1,\ldots,c_N$ are non-zero real numbers.
\vskip 2mm

\noindent For every $i\in \aN$, the $(\omega_i^{(n)})_{n\in \Z}$ are sequences of independent and identically distributed \emph{i.i.d.} random variables on a complete probability space $(\widetilde{\Omega},\widetilde{\mathcal{A}}, \widetilde{\mathsf{P}})$, of common law $\nu$ such that $\{0,1\} \subset \supp \nu$ and $\supp \nu$ is bounded. In particular, the $\omega_i^{(n)}$'s can be Bernoulli random variables. The family $\{\Hl \}_{\omega \in \Omega}$ is a family of random operators indexed by the product space
$$(\Omega,\mathcal{A},\mathsf{P})=\left(\otimes_{n\in \Z}\widetilde{\Omega}^{\otimes N},\otimes_{n\in \Z}\widetilde{\mathcal{A}}^{\otimes N}, \otimes_{n\in \Z}\widetilde{\mathsf{P}}^{\otimes N}\right).$$

\noindent We also set, for every $n\in \Z$, {\small $\omega^{(n)}=(\omega_1^{(n)},\ldots,\omega_N^{(n)})$}, which is a random variable on {\small $(\widetilde{\Omega}^{\otimes N},\widetilde{\mathcal{A}}^{\otimes N}, \widetilde{\mathsf{P}}^{\otimes N})$} of law $\nu^{\otimes N}$. The expectancy against $\mathsf{P}$ will be denoted by $\E(.)$.
\vskip 2mm

\noindent As a bounded perturbation of $-\frac{\dd^2}{\dd x^2}\otimes I_{\mathrm{N}}$, the operator $\Hl$ is self-adjoint on the Sobolev space  $H^2(\R)\otimes \R^N$ and thus, for every $\omega\in \Omega$, the spectrum of $\Hl$, $\sigma(\Hl)$, is included in $\R$. Moreover, because of the periodicity in law of the random potential of $\Hl$, the family $\{\Hl \}_{\omega \in \Omega}$ is $\ell \Z$-ergodic. Thus, there exists $\Sigma \subset \R$ such that, for $\mathsf{P}$-almost every $\omega \in \Omega$, $\Sigma=\sigma(\Hl)$. There also exist $\Sigma_{\mathrm{pp}}$, $\Sigma_{\mathrm{ac}}$ and $\Sigma_{\mathrm{sc}}$, subsets of $\R$, such that, for $\mathsf{P}$-almost every $\omega \in \Omega$, $\Sigma_{\mathrm{pp}}=\sigma_{\mathrm{pp}}(\Hl)$, $\Sigma_{\mathrm{ac}}=\sigma_{\mathrm{ac}}(\Hl)$ and $\Sigma_{\mathrm{sc}}=\sigma_{\mathrm{sc}}(\Hl)$, respectively the pure point, absolutely continuous and singular continuous spectrum of $\Hl$.
\vskip 1mm

\noindent Our main result will be about localization properties of $\Hl$. Before stating it, we give the definitions of both exponential localization and dynamical localization for $\Hl$. We denote by $E_{\omega}(.)$ the spectral projection of the self-adjoint operator $\Hl$ and the $L^2$-norm is written as $||\ ||$.

\begin{defi}\label{def_localization}
Let $I\subset \R$ be an open interval. We say that :
\begin{itemize}
\item[(i)] $\Hl$ exhibits \emph{exponential localization} (EL) in $I$, if it has pure point spectrum in $I$ (\emph{i.e.}, $\Sigma\cap I=\Sigma_{\mathrm{pp}}\cap I$ and $\Sigma_{\mathrm{ac}}\cap I=\Sigma_{\mathrm{sc}}\cap I=\emptyset$) and, for $\mathsf{P}$-almost every $\omega\in \Omega$, the eigenfunctions of $\Hl$ with eigenvalues in $I$ decay exponentially in the $L^2$-sense (\emph{i.e.}, there exist $C$ and $m>0$ such that $||\mathbf{1}_{[x-\ell,x+\ell]} \psi || \leq C\ee^{-m|x|}$ for $\psi$ an eigenfunction of $\Hl$) ;
\item[(ii)] $\Hl$ exhibits \emph{strong dynamical localization} (SDL) in $I$, if $\Sigma \cap I \neq \emptyset$ and, for each compact interval $\tilde{I} \subset I$ and $\psi \in L^2(\R)\otimes \R^N$ with compact support, we have,
$$\forall n\geq 0,\ \E\left( \sup_{t\in \R} \left|\left| \left(\sqrt{1+|x|^2}\right)^{\frac{n}{2}} E_{\omega} (\tilde{I}) \ee^{-\ii t \Hl} \psi\right|\right|^2 \right) <\infty \ .$$
\end{itemize}
\end{defi}
\vskip 3mm

\noindent Before stating our main results, we need to introduce some more notations. Let $\SpN$ denote the group of $2\mathrm{N}\times 2\mathrm{N}$ real symplectic matrices. It is the subgroup of $\mathrm{GL}_{2\mathrm{N}}(\R)$ of matrices $M$ satisfying $$^tMJM=J,$$ 
where $J$ is the matrix of order $2N$ defined by $J=\bigl(\begin{smallmatrix}
0 & -I_{N} \\
I_{N} & 0
\end{smallmatrix}\bigr)$. Let $\mathcal{O}$ be the neighborhood of $I_{2\mathrm{N}}$ in $\SpN$ given by Theorem \ref{thm_BG03} applied to $G=\SpN$.

\noindent We set :
\begin{equation}\label{eq_def_dlogO}
\dlO=\max\{ R>0\ |\ B(0,R)\subset \log\,\mathcal{O} \},
\end{equation}
where $B(0,R)$ is the open ball, centered on $0$ and of radius $R>0$, for the topology induced on the Lie algebra $\spN$ of $\SpN$ by the matrix norm induced by the euclidean norm on $R^{2N}$. 

\noindent For {\small $\omO=(\omega_1^{(0)},\ldots,\omega_N^{(0)})\in \{0,1\}^N$}, let 
$$M_{\omO}(0,V) = V + \mathrm{diag}(c_1 \omega_1^{(0)},\ldots,c_N \omega_N^{(0)}).$$
As $M_{\omO}(0,V)\in \SN$, it has $\lambda_1^{\omO},\ldots ,\lambda_N^{\omO}$ as real eigenvalues. We set,
\begin{equation}\label{eq_def_lambda_minmax}
\lambda_{\mathrm{min}}=\min_{\omO\in \{ 0,1\}^N} \min_{1\leq i\leq N} \lambda_i^{\omO},\qquad \lambda_{\mathrm{max}}=\max_{\omO\in \{ 0,1\}^N} \max_{1\leq i\leq N} \lambda_i^{\omO}
\end{equation}
and $\delta=\frac{\lambda_{\mathrm{max}}-\lambda_{\mathrm{min}}}{2}$. We also set 
\begin{equation}\label{eq_def_lCV}
\ell_C:=\ell_C(N,V)=\min \left( 1, \frac{\dlO}{\delta}\right)
\end{equation}
and, for every $\ell\in (0,\ell_C)$,
\begin{equation}\label{eq_def_INVl}
I(N,V,\ell)=\left[ \lambda_{\mathrm{max}} - \frac{\dlO}{\ell},\lambda_{\mathrm{min}} + \frac{\dlO}{\ell} \right].
\end{equation}

\noindent We remark that, as $\ell$ tends to $0^+$, $I(N,V,\ell)$ tends to the whole real line. We can now state our main result.

\begin{thm}\label{thm_localization}
For almost every $V \in \SN$, there exists a finite set $\SV \subset \R$ such that, for every $\ell\in (0,\ell_C)$, if $I\subset I(N,V,\ell) \setminus \SV$ is an open interval with $\Sigma \cap I \neq \emptyset$, then $\Hl$ exhibits (EL) and (SDL) on $I$.
\end{thm}
\vskip 3mm

\noindent Here, ``almost every'' is considered according to the Lebesgue measure on $\SN$ identified to $\R^{\frac{N(N+1)}{2}}$. We also remark that, as $I(N,V,\ell)$ tends to $\R$ when $\ell$ tends to $0^+$, taking $\ell\in (0,\ell_C)$ small enough ensure that we can always find a non-trivial open interval $I\subset  I(N,V,\ell) \setminus \SV$ such that $\Sigma \cap I \neq \emptyset$.
\vskip 3mm

\noindent This theorem will follow from the next proposition. For $E\in \R$, let $G(E)$ be the F\"urstenberg group associated to $\Hl$ (see Definition \ref{def_GE}).

\begin{prop}\label{prop_GE_SpN}
For almost every $V \in \SN$, there exists a finite set $\SV \subset \R$ such that, for every $\ell\in (0,\ell_C)$,
$$\forall E\in I(N,V,\ell) \setminus \SV,\ G(E)=\SpN.$$
\end{prop}
\vskip 2mm

\noindent In particular, Proposition \ref{prop_GE_SpN} will imply the separability of the Lyapunov exponents of $\Hl$ (see Definition \ref{def_lyap_exp}) and the absence of absolutely continuous spectrum in $I(N,V,\ell)$, for $\ell\in (0,\ell_C)$. 

\begin{cor}\label{cor_lyap_exp}
For almost every $V \in \SN$, there exists a finite set $\SV \subset \R$ such that, for every $\ell\in (0,\ell_C)$, the $N$ positive Lyapunov exponents of $\Hl$, $\gamma_1(E),\ldots,\gamma_N(E)$, verify: 
\begin{equation}\label{eq_cor_lyap_exp}
\forall E\in I(N,V,\ell)\setminus \SV, \quad \gamma_{1}(E)>\cdots > \gamma_{N}(E)>0.
\end{equation}
Therefore, $\Hl$ has no absolutely continuous spectrum in $I(N,V,\ell)$, \emph{i.e.}, for every $\ell\in (0,\ell_C)$, $\Sigma_{\mathrm{ac}} \cap I(N,V,\ell)=\emptyset$.
\end{cor}
\vskip 5mm

It is already known in the scalar-valued case (corresponding here to $N=1$) that, away from a discrete set $\mathcal{S}$ of critical energies, there is exponential localization and strong dynamical localization on every compact interval $I\subset \R\setminus \mathcal{S}$ with $\Sigma \cap I \neq \emptyset$ (see \cite{DSS02}). But, in dimension $d$ higher than $2$, the question of the localization remains mostly open for Anderson-Bernoulli models. Such an Anderson-Bernoulli model is given by a family of random operators of the form
\begin{equation}\label{eq_def_H}
\Ho=-\Delta_{d} + \sum_{n\in \Z^d}\omega_n V(x-n),
\end{equation}
acting on $L^2(\R^d)\otimes \R$, where $V$ is supported in $[0,1]^d$ and the $\omega_n$ are \emph{i.i.d.} Bernoulli random variables. Since \cite{BK05}, it is known that there is exponential localization at the bottom of the almost sure spectrum of $\Ho$. In dimension $d\geq 3$, it is commonly conjectured that for high energies, there exist extended states, as for dimension $d=2$ it is conjectured that there is localization at every energies, except maybe those in a discrete set.

To tackle the question of localization for $d=2$, we can start by looking at a slightly simpler model, a continuous strip $\R\times [0,1]$ in $\R^2$. This model is given by the restriction $H_{\mathrm{cs}}(\omega)$ of $\Ho$ to $L^2(\R \times [0,1])$, with Dirichlet boundary conditions on $\R \times \{0 \}$ and $\R \times \{1\}$. This model can be used to study transport properties of nanoconductors so it is also of physical interest. The question of the localization at all energies for $H_{\mathrm{cs}}(\omega)$ present difficulties of the same level as for $\Ho$, mostly due to the PDE's nature of the problem in both cases. But, for $H_{\mathrm{cs}}(\omega)$, we have a possible approach by operating a discretization in the bounded direction of the strip. This can be performed by first applying discrete  Fourier transform in the second variable corresponding to the bounded direction, which leads to a quasi one-dimensional model with an infinite size matrix for potential. Then, by applying a cut-off in the space of Fourier frequencies, we obtain a quasi one-dimensional model with a matrix of finite order $N$ for potential, acting on $L^2(\R)\otimes \R^N$,  with $N\geq 1$ an integer. It turns the nature of the initial PDE's problem to an ODE's one, which allows to use formalism such as transfer matrices and Lyapunov exponents. The model (\ref{eq_model_Hl}) we are looking at here is not exactly the one obtained by this discretization procedure, but the understanding of localization properties for (\ref{eq_model_Hl}) should lead us to the same understanding for the discretize operator obtained from $H_{\mathrm{cs}}(\omega)$.
\vskip 3mm

We finish this introduction by giving the outline of the article. In Section \ref{sec_trans_mat}, we present the formalism of transfer matrices and compute them for $\Hl$. We also define the Lyapunov exponents and the F\"urstenberg group associated to $\Hl$. In Section \ref{sec_lie_algebra}, we study the Lie algebra generated by the matrices $X_{\omO}(E,V)$ defined at (\ref{eq_def_XEV}). In this section we also prove the genericity argument and we construct the finite set $\SV$ of Theorem \ref{thm_localization}, Proposition \ref{prop_GE_SpN} and Corollary \ref{cor_lyap_exp}. This genericity argument is mostly based upon algebraic geometry considerations and the Lebesgue measure of affine algebraic manifolds. In Section \ref{sec_proof_prop1}, we prove Proposition \ref{prop_GE_SpN} and Corollary \ref{cor_lyap_exp} and we explicitely construct $\ell_C$ and $I(N,V,\ell)$ for $\ell\in (0,\ell_C)$. The proofs of this section are based upon a general result on Lie groups due to Breuillard and Gelander (see Theorem \ref{thm_BG03}). In Section \ref{sec_proof_loc}, we recall localization results of \cite{boumaza_mpag2} and we deduce from them the proof of Theorem \ref{thm_localization}. Finally, in Section \ref{sec_ids}, we state a result of existence and regularity of the integrated density of states associated to $\Hl$.
\vskip 3mm

The general idea of the proof of Theorem \ref{thm_localization} can be briefly sketched. First we change the initial spectral and dynamical problem of the localization into a topological problem on proving that a Lie group with a finite number of generators is dense in the real symplectic group $\SpN$, which is the statement of Proposition \ref{prop_GE_SpN}. Then, we use the general criterion on Lie groups of Breuillard and Gelander to transform this topological problem into a purely algebraic problem on generating the Lie algebra $\spN$. The algebraic nature of the objects we are considering at this last step allows us to prove a generic result on $V$ and the finiteness of the set $\SV$ of critical energies.

\section{Transfer matrices and the F\"urstenberg group}\label{sec_trans_mat}

Let $E\in \R$. We want to understand the exponential asymptotic behaviour of a solution $u:\R\to \R^N$ of the second order differential system
\begin{equation}\label{eq_system_order2}
\Hl u=Eu.
\end{equation}
For this, we transform (\ref{eq_system_order2}) into an Hamiltonian differential system of order $1$ and we introduce the transfer matrix $T_{\omega^{(n)}}(E)$ of $\Hl$ from $\ell n$ to $\ell (n+1)$ which maps a solution $(u,u')$ of the order $1$ system at time $\ell n$ to the solution at time $\ell (n+1)$. The transfer matrix  $T_{\omega^{(n)}}(E)$ is therefore defined by the relation
\begin{equation}\label{eq_def_transfer_mat}
\forall n\in \Z,\ \left( \begin{array}{c}
u(\ell(n+1)) \\
u'(\ell(n+1)) 
\end{array} \right) = T_{\omega^{(n)}}(E) \left( \begin{array}{c}
u(\ell n) \\
u'(\ell n) 
\end{array} \right).
\end{equation}
As the system of order $1$ is Hamiltonian, the transfer matrix $T_{\omega^{(n)}}(E)$ lies into the symplectic group $\SpN$. The sequence $(T_{\omega^{(n)}}(E))_{n\in \Z}$ is also a sequence of \emph{i.i.d.} symplectic matrices because of the \emph{i.i.d.} character of the $\omega_i^{(n)}$'s and the non-overlapping of these random variables. By iterating the relation (\ref{eq_def_transfer_mat}) we get the asymptotic behaviour of $(u,u')$. To get the \emph{exponential} asymptotic behaviour of $(u,u')$ we can define the exponential growth (or decay) exponents of the product of random matrices $T_{\omega^{(n-1)}}(E)\ldots T_{\omega^{(0)}}(E)$.

\begin{defi}\label{def_lyap_exp}
Let $E\in \R$. The Lyapunov exponents $\gamma_{1}(E),\ldots,\gamma_{2N}(E)$, associated to the sequence $(T_{\omega^{(n)}}(E))_{n\in \Z}$, are defined inductively by
\begin{equation}\label{eq_def_lyap_exp}
\sum_{i=1}^{p} \gamma_{i}(E) = \lim_{n \to \infty} \frac{1}{n}
\mathbb{E}(\log ||\wedge^{p} (T_{\omega^{(n-1)}}(E)\ldots T_{\omega^{(0)}}(E))||),
\end{equation}
for every $p\in \{1,\ldots,2N\}$.
\end{defi}

\noindent Here, $\wedge^{p} M$ denotes the $p$th exterior power of the matrix $M$, acting on the $p$th exterior power of $\R^{2N}$. One has $\gamma_{1}(E)\geq \ldots \geq \gamma_{2N}(E)$. Moreover, due to the symplecticity of the random matrices $T_{\omega^{(n)}}(E)$, we have the symmetry property $\gamma_{2N-i+1}= -\gamma_{i}$, for every $i \in \aN$. Thus, we will only have to study the $N$ first Lyapunov exponents to obtain Corollary \ref{cor_lyap_exp}. To prove the separability of the Lyapunov exponents, we introduce the group which contains all the different products of transfer matrices, the so-called F\"urstenberg group.

\begin{defi}\label{def_GE}
For every $E\in \R$, the \emph{F\"urstenberg group} of $\Hl$ is defined by 
$$G(E)=\overline{<\supp \mu_E>},$$
where $\mu_E$ is the common distribution of the $T_{\omega^{(n)}}(E)$ and the closure is taken for the usual topology in $\SpN$.
\end{defi}

\noindent As the $T_{\omega^{(n)}}(E)$ are \emph{i.i.d.}, $\mu_E=(T_{\omO}(E))_{*}\, \nu^{\otimes N}$ and we have the internal description of $G(E)$ :
\begin{equation}\label{eq_GE_intern}
\forall E\in \R,\ G(E)=\overline{<T_{\omO} (E)\ |\ \omO \in \supp \nu^{\otimes N}>}.
\end{equation}
As $\{0,1\}\subset \supp \nu$, we also have 
\begin{equation}\label{eq_GE_inclusion}
\overline{<T_{\omO} (E)\ |\ \omO \in \{0,1\}^N >}\subset G(E).
\end{equation}
We will denote be $G_{\{0,1\}}(E)$ the subgroup of $G(E)$ with $2^N$ generators :
\begin{equation}\label{eq_def_G01E}
G_{\{0,1\}}(E)=<T_{\omO} (E)\ |\ \omO \in \{0,1\}^N >.
\end{equation}
\vskip 2mm

\noindent In Section \ref{sec_proof_prop1}, we will prove that, for almost every $V\in \SN$ and for all $E\in \R$ except those in a finite set, $G_{\{0,1\}}(E)$ is dense in $\SpN$.
\vskip 3mm

\noindent We finish this section by giving the explicit form of the transfer matrices $T_{\omega^{(n)}}(E)$. Let $V\in \SN$, $E\in \R$, $n\in \Z$ and $\omega^{(n)}\in \tilde{\Omega}^{\otimes N}$. We set : 
\begin{equation}\label{eq_def_ME}
M_{\omega^{(n)}}(E,V)=V+\mathrm{diag}(c_1 \omega_1^{(n)},\ldots,c_N \omega_N^{(n)})-EI_{\mathrm{N}}. 
\end{equation}
\noindent Then, we set the following matrix of the Lie algebra $\spN$, 
\begin{equation}\label{eq_def_XEV}
X_{\omega^{(n)}}(E,V)=\left( \begin{array}{cc}
0 & I_{\mathrm{N}} \\
M_{\omega^{(n)}}(E,V) & 0
\end{array}\right) \in \spN \subset \mathcal{M}_{\mathrm{2N}}(\R).
\end{equation}
By solving the constant coefficients system (\ref{eq_system_order2}) on $[\ell n,\ell (n+1)]$, we have :
\begin{equation}\label{eq_expression_mat_transfer}
\forall \ell >0,\ \forall n\in \Z,\ \forall V\in \SN,\ \forall E\in \R,\ T_{\omega^{(n)}}(E)=\exp\left(\ell X_{\omega^{(n)}}(E,V)\right).
\end{equation}
\vskip 2mm

\noindent It is important here to notice that $T_{\omega^{(n)}}(E)$ is the exponential of a matrix, as it will be crucial to be able to apply Theorem \ref{thm_BG03} to the subgroup $G_{\{0,1\}}(E)$.

\section{The Lie algebra generated by {\small $\{X_{\omO}(E,V)\ |\ \omO\in \{0,1\}^N\}$}}\label{sec_lie_algebra}

In this section we will present in details the proof of the genericity argument needed to prove Proposition \ref{prop_GE_SpN} and Theorem \ref{thm_localization}. We start by looking at the geometry of the set of $k$-uples in $\spN$ which do not generates $\spN$ in the sense of Lie algebras.

\begin{lem}\label{lem_geom}
Let $k\in \N^*$ and
{\small \begin{equation}\label{eq_def_Vk}
\mathcal{V}_k=\left\{ (X_1,\ldots,X_k)\in (\spN)^k \ |\ (X_1,\ldots,X_k)\ \mathrm{does}\ \mathrm{not}\ \mathrm{generate}\ \spN \right\}.
\end{equation}}
Then, there exist $Q_{r_1},\ldots,Q_{r_k}\in \R[(\spN)^k]$ such that :
{\small \begin{equation}\label{eq_lem_Vk}
\mathcal{V}_k=\left\{ (X_1,\ldots,X_k)\in (\spN)^k \ |\ Q_{r_1}(X_1,\ldots,X_k)=0,\ldots, Q_{r_k}(X_1,\ldots,X_k)=0 \right\}.
\end{equation}}
\end{lem}
\vskip 1mm

\noindent Thus, $\mathcal{V}_k$ is the affine algebraic manifold of $\{Q_{r_1},\ldots,Q_{r_k}\}$ which will be denoted by $V(\{Q_{r_1},\ldots,Q_{r_k} \})$. We will also use the identification :
\begin{equation}\label{eq_id_poly}
\R[(\spN)^k]\simeq \R[T_1,\ldots,T_{k(2N^2 +N)}]. 
\end{equation}

\begin{proof}
Let $(X_1,\ldots,X_k)\in (\spN)^k$ and $\mathrm{Lie}\{X_1,\ldots ,X_k\}$ be the Lie algebra generated by $X_1,\ldots, X_k$. If we denote by $\{Y_1,\ldots,Y_l,\ldots \}$ the countable set of all the successives brackets constructed from $\{X_1,\ldots,X_k\}$, we have
\begin{equation}\label{eq_lem_Vk_1}
\mathrm{Lie}\{X_1,\ldots ,X_k\}=\mathrm{span}(\{Y_1,\ldots,Y_l,\ldots \}),
\end{equation}
the vector space spanned by $\{Y_1,\ldots,Y_l,\ldots \}$. Then we have :
\begin{equation}\label{eq_lem_Vk_2}
\mathrm{Lie}\{X_1,\ldots ,X_k\}\neq \spN \ \Leftrightarrow \ \mathrm{rk}(\{Y_1,\ldots,Y_l,\ldots \}) < 2N^2+N,
\end{equation}
as $\mathrm{dim}\; \spN =2N^2+N$. At each $Y_l\in \spN$ we associate $\tilde{Y}_l \in \R^{2N^2+N}$ whose coefficients are those which define the matrix $Y_l$. The coefficients of $\tilde{Y}_l$ are polynomial in the $k(2N^2+N)$ coefficients which define the matrices $X_1,\ldots X_k$. For $m\in (\N^*)^{2N^2+N}$, we set
\begin{equation}\label{eq_lem_Vk_3}
Q_m(X_1,\ldots,X_k)=\det (\tilde{Y}_{m_1},\ldots, \tilde{Y}_{m_{2N^2+N}}) \in \R[(\spN)^k].
\end{equation}
Then,
\begin{equation}\label{eq_lem_Vk_4}
\mathrm{rk}(\{Y_1,\ldots,Y_l,\ldots \}) < 2N^2+N\ \Leftrightarrow \ \forall m\in (\N^*)^{2N^2+N}, Q_m(X_1,\ldots X_k)=0.
\end{equation}
Thus,
\begin{equation}\label{eq_lem_Vk_5}
\mathcal{V}_k = \bigcap_{ m\in (\N^*)^{2N^2+N}} \left\{  (X_1,\ldots,X_k)\in (\spN)^k \ |\ Q_m(X_1,\ldots,X_k)=0 \right\}.
\end{equation}
With the definition of the affine algebraic manifold, we can rewrite (\ref{eq_lem_Vk_5}) as :
\begin{equation}\label{eq_lem_Vk_6}
\mathcal{V}_k = V(\{ Q_m\ |\  m\in (\N^*)^{2N^2+N} \}).
\end{equation}
But, if $I(\{ Q_m\ |\  m\in (\N^*)^{2N^2+N} \})$ denote the ideal generated by the family $\{ Q_m\ |\  m\in (\N^*)^{2N^2+N} \}$, we have :
\begin{equation}\label{eq_lem_Vk_7}
V(\{ Q_m\ |\  m\in (\N^*)^{2N^2+N} \})=V(I(\{ Q_m\ |\  m\in (\N^*)^{2N^2+N} \})).
\end{equation}
As the ring $\R[T_1,\ldots,T_{k(2N^2 +N)}]$ is Noetherian, $I(\{ Q_m\ |\  m\in (\N^*)^{2N^2+N} \})$ is of finite type, \emph{i.e.} there exist $r_1,\ldots,r_k\in (\N^*)^{2N^2+N}$ such that,
\begin{equation}\label{eq_lem_Vk_8}
I(\{ Q_m\ |\  m\in (\N^*)^{2N^2+N} \})=I(\{Q_{r_1},\ldots,Q_{r_k} \}).
\end{equation}
Finally,
\begin{equation}\label{eq_lem_Vk_9}
\mathcal{V}_k = V(I(\{Q_{r_1},\ldots,Q_{r_k} \}))=V(\{Q_{r_1},\ldots,Q_{r_k} \}).
\end{equation}
\end{proof}
\vskip 3mm

\noindent For $E\in \R$ and $V\in \SN$, we will reindex the family {\small $\{ X_{\omO}(E,V)\}_{\omO \in \{0,1\}^N}$ } as $(X_1(E,V),\ldots,X_{2^N}(E,V))$. Let $E\in \R$ be fixed and let
\begin{equation}\label{eq_def_VE}
\mathcal{V}_{(E)}=\left\{ V\in \SN \ |\ (X_1(E,V),\ldots,X_{2^N}(E,V))\ \mathrm{does}\ \mathrm{not}\ \mathrm{generate}\ \spN \right\}.
\end{equation}

\begin{lem}\label{lem_VE_Leb}
We have, $\mathrm{Leb}_{\frac{N(N+1)}{2}}(\mathcal{V}_{(E)})=0$.
\end{lem}

\begin{proof}
Let
\begin{equation}\label{eq_lem_VE_1}
f_E\ :\ \begin{array}{ccl}
\SN & \to & (\spN)^{2^N} \\
V & \mapsto & (X_1(E,V),\ldots, X_{2^N}(E,V))
\end{array}.
\end{equation} 
Then $f_E$ is polynomial in the $\frac{N(N+1)}{2}$ coefficients which define $V$. Indeed, we can identify $\SN \simeq \R^{\frac{N(N+1)}{2}}$ and $(\spN)^{2^N}\simeq \R^{2^N(2N^2+N)}$ and, after this identification, $f_E$ has each of its $2^N(2N^2+N)$ components polynomial in $\R[T_1,\ldots,T_{\frac{N(N+1)}{2}}]$.

\noindent We have :
\begin{equation}\label{eq_lem_VE_2}
\mathcal{V}_{(E)}=f_E^{-1}(\mathcal{V}_{2^N}).
\end{equation}
Then, by Lemma \ref{lem_geom}, $V\in \mathcal{V}_{(E)}$ if and only if  
$$Q_{r_1}(X_1(E,V),\ldots ,X_{2^N}(E,V))=0,\ldots, Q_{r_{2^N}}(X_1(E,V),\ldots ,X_{2^N}(E,V))=0,$$ 
which can be rewrite
\begin{equation}\label{eq_lem_VE_4}
V\in \mathcal{V}_{(E)}\ \Leftrightarrow \ (Q_{r_1}\circ f_E)(V)=0,\ldots,(Q_{r_{2^N}}\circ f_E)(V)=0.
\end{equation}
But, we can prove that, if $V_0$ is the tridiagonal matrix with zeros on the diagonal and all coefficients on its upper and lower diagonals equal to $1$, then, for any $E\in \R$, $V_0\notin \mathcal{V}_{(E)}$ (see \cite[Lemma 3]{boumaza_mpag2}). Thus, there exists $i_0\in \{r_1,\ldots,r_{2^N} \}$ such that $(Q_{i_0} \circ f_E)(V_0) \neq 0$ and, as the function $Q_{i_0} \circ f_E$ is polynomial and do not vanish identically,
\begin{equation}\label{eq_lem_VE_5}
\mathrm{Leb}_{\frac{N(N+1)}{2}}\left( \{ V\in \SN \ |\ (Q_{i_0} \circ f_E)(V)=0) \} \right)=0,
\end{equation}
and, by inclusion,
\begin{equation}\label{eq_lem_VE_6}
\mathrm{Leb}_{\frac{N(N+1)}{2}}(\mathcal{V}_{(E)})=0.
\end{equation}
\end{proof}
\vskip 3mm

\noindent Finally, we can introduce the set :
\begin{equation}\label{eq_def_V}
\mathcal{V}  =  \bigcap_{E\in \R}  \mathcal{V}_{(E)} \qquad \qquad\qquad\qquad\qquad\qquad\qquad\qquad\qquad \qquad\qquad\qquad \qquad
\end{equation}
$$=\left\{ V\in \SN \ |\ \forall E\in \R,\ (X_1(E,V),\ldots,X_{2^N}(E,V))\ \mathrm{does}\ \mathrm{not}\ \mathrm{generate}\ \spN \right\}.$$

\noindent Then, by Lemma \ref{lem_VE_Leb} and by inclusion, we have :
\begin{equation}\label{eq_V_Leb_zero}
\mathrm{Leb}_{\frac{N(N+1)}{2}}(\mathcal{V})=0.
\end{equation}
Now we can prove the last result of this section.

\begin{lem}\label{lem_SV_finite}
For any $V\in \SN \setminus \mathcal{V}$, there exists a finite set $\SV \subset \R$ such that : 
$$\forall E\in \R\setminus \SV,\ (X_1(E,V),\ldots,X_{2^N}(E,V))\ \mathrm{generates}\ \spN.$$
\end{lem}

\begin{proof}
Let $V\in \SN \setminus \mathcal{V}$. Then, there exists $E_0 \in \R$ such that the family  $(X_1(E_0,V),\ldots,X_{2^N}(E_0,V))$ generates $\spN$. Thus, there exists $i_0 \in  \{r_1,\ldots,r_{2^N} \}$ such that $(Q_{i_0} \circ f)(E_0,V) \neq 0$, where 
\begin{equation}\label{eq_lem_SV_1}
f\ :\ \begin{array}{ccl}
\R \times \SN & \to & (\spN)^{2^N} \\
(E,V) & \mapsto & (X_1(E,V),\ldots, X_{2^N}(E,V))
\end{array}.
\end{equation}
But, for $V$ fixed, $E\mapsto (Q_{i_0} \circ f)(E,V)$ is polynomial and, as it is not identically vanishing, it has only a finite set $\SV$ of roots. Thus, we have :
\begin{equation}\label{eq_lem_SV_2}
\forall E\in \R \setminus \SV,\ (Q_{i_0} \circ f)(E,V) \neq 0,
\end{equation}
which is equivalent to :
\begin{equation}\label{eq_lem_SV_3}
\forall E\in \R \setminus \SV,\ (X_1(E,V),\ldots, X_{2^N}(E,V)) \notin \mathcal{V}_{2^N}.
\end{equation}
\end{proof}
\vskip 3mm

\noindent With this Lemma \ref{lem_SV_finite} we are now able to prove Proposition \ref{prop_GE_SpN}.

\section{Proof of Proposition \ref{prop_GE_SpN} and Corollary \ref{cor_lyap_exp}}\label{sec_proof_prop1}

The proof of Proposition \ref{prop_GE_SpN} is based upon a general criterion of density in semisimple Lie groups due to Breuillard and Gelander.

\begin{thm}[\cite{BG03}, Theorem 2.1]\label{thm_BG03}
Let $G$ be a real, connected, semisimple Lie group, whose Lie algebra is $\mathfrak{g}$. Then, there is a neighborhood $\mathcal{O}$ of $1$ in $G$, on which $\log=\exp^{-1}$ is a well defined diffeomorphism, such that $g_{1},\ldots,g_{m}\in \mathcal{O}$ generate a dense subgroup whenever $\log g_{1},\ldots,\log g_{m}$ generate $\mathfrak{g}$.
\end{thm}

This criterion, applied to $G=\SpN$, gives us the outline of the proof of Proposition \ref{prop_GE_SpN} :
\begin{itemize}
\item[(i)] We construct $\ell_C$ and $I(N,V,\ell)$ such that, for $\ell \in (0,\ell_C)$ and $E\in I(N,V,\ell)$, $T_{\omO}(E)\in \mathcal{O}$, for every $\omO \in \{0,1\}^N$.
\item[(ii)] We compute $\log T_{\omO}(E)$.
\item[(iii)] We justify that $\mathrm{Lie}\{\log T_{\omO}(E)\ |\ \omO \in \{0,1\}^N \}=\spN$ for $V\in \SN \setminus \mathcal{V}$ and $E\in I(N,V,\ell) \setminus \SV$.
\item[(iv)] We deduce that $G_{\{0,1\}}(E)$ is dense for the usual topology in $\SpN$, for $E\in I(N,V,\ell) \setminus \SV$.
\end{itemize}
\vskip 3mm

\begin{proof}
We fix $V\in \SN \setminus \mathcal{V}$. We start be constructing $\ell_C$ and, for $\ell \in (0,\ell_C)$, the interval $I(N,V,\ell)$ as given in (\ref{eq_def_lCV}) and (\ref{eq_def_INVl}). Now, let $\lambda_1^{\omO}$,$\ldots$, $\lambda_N^{\omO}$ be the real eigenvalues of $M_{\omO}(0,V)$ (see (\ref{eq_def_ME})). Then, the eigenvalues of $X_{\omO}(E,V) ^tX_{\omO}(E,V)$ are $1$, $(\lambda_1^{\omO}-E)^2$, $\ldots$, $(\lambda_N^{\omO}-E)^2$, thus :
\begin{equation}\label{eq_prop1_1}
||X_{\omO}(E,V)||=\max \left(1, \max_{1\leq i \leq N} |\lambda_i^{\omO}-E|\right),
\end{equation}
where $||\ ||$ is the matrix norm associated to the euclidian norm on $\R^{2N}$.

\noindent Let $\mathcal{O}$ be the neighborhood of the identity given by Theorem \ref{thm_BG03} applied to the group $G=\SpN$. Then, for $\dlO$ as defined in (\ref{eq_def_dlogO}), we take $\ell \leq \dlO$ and we set  $r_{\ell}=\frac{1}{\ell}\dlO \geq 1$. If we set
\begin{equation}\label{eq_prop1_2}
I(N,V,\ell)=\left\{E\in \R\ \bigg|\  \max \left(1,\max_{\omO\in \{ 0,1\}^N} \max_{1\leq i\leq N} |\lambda_i^{\omO}-E|\right) \leq r_{\ell} \right\},
\end{equation}
as $r_{\ell} \geq 1$,
\begin{equation}\label{eq_prop1_3}
I(N,V,\ell)=\bigcap_{\omO\in \{ 0,1\}^N} \bigcap_{1\leq i\leq N} [\lambda_i^{\omO}-r_{\ell}, \lambda_i^{\omO}+r_{\ell}].
\end{equation}

\noindent Let $\lambda_{\mathrm{min}}$, $\lambda_{\mathrm{max}}$ and $\delta$ be as in (\ref{eq_def_lambda_minmax}). If $\delta <r_{\ell}$ then $I(N,V,\ell)\neq \emptyset$ and we have 
\begin{equation}\label{eq_prop1_4}
I(N,V,\ell)=[\lambda_{\mathrm{max}}-r_{\ell},\lambda_{\mathrm{min}}+r_{\ell}],
\end{equation}
which is the definition we took in (\ref{eq_def_INVl}). This interval is centered in $\frac{\lambda_{\mathrm{min}}+\lambda_{\mathrm{max}}}{2}$ and is of length $2r_{\ell}-2\delta >0$, which tends to $+\infty$ when $\ell$ tends to $0^+$. We also note that $\lambda_{\mathrm{min}}$, $\lambda_{\mathrm{max}}$ and $\dlO$ depend only on $N$ and $V$ and thus $I(N,V,\ell)$ depends only on $N$, $V$ and $\ell$. Finally, the condition $\delta <r_{\ell}$, which ensures that $I(N,V,\ell)\neq \emptyset$, is equivalent to
$$0 < \ell <\frac{\dlO}{\delta}=\ell_C(N,V).$$
\noindent So, we have just proved that, 
\begin{equation}\label{eq_prop1_5}
\forall \ell \in (0,\ell_C),\ \forall E\in I(N,V,\ell),\ \forall \omO \in \{0,1\}^N,\ 0<\ell ||X_{\omO}(E,V)|| \leq \dlO.
\end{equation}
Thus, for every $\ell \in (0,\ell_C)$ and every $E\in I(N,V,\ell)$,
\begin{equation}\label{eq_prop1_6}
\forall \omO \in \{0,1\}^N,\ \ell X_{\omO}(E,V) \in \log \mathcal{O}.
\end{equation}
From this, we deduce that,
\begin{equation}\label{eq_prop1_7}
\forall \ell \in (0,\ell_C),\ \forall E\in I(N,V,\ell),\ \forall \omO \in \{0,1\}^N,\ T_{\omO}(E) \in \mathcal{O}.
\end{equation}
We actually get more from (\ref{eq_prop1_6}). As \emph{exp} is a diffeomorphism from $\log \mathcal{O}$ into $\mathcal{O}$, we also have :
\begin{equation}\label{eq_prop1_8}
\forall \ell \in (0,\ell_C),\ \forall E\in I(N,V,\ell),\ \forall \omO \in \{0,1\}^N,\ \log\,T_{\omO}(E)=\ell X_{\omO}(E,V).
\end{equation}
But, from the beginning, we choosed $V\in \SN \setminus \mathcal{V}$ and, by Lemma \ref{lem_SV_finite}, there exists $\SV \subset \R$ finite such that
\begin{equation}\label{eq_prop1_9}
\forall E\in \R\setminus \SV,\ \mathrm{Lie} \{ X_{\omO}(E,V) \ |\ \omO \in \{0,1\}^N \}=\spN.
\end{equation}
Now, by (\ref{eq_prop1_8}) and (\ref{eq_prop1_9}), as $\ell \in (0,\ell_C)$ is different from $0$,
\begin{equation}\label{eq_prop1_10}
\forall \ell \in (0,\ell_C),\ \forall E\in I(N,V,\ell)\setminus \SV,\ \mathrm{Lie} \{ \log\, T_{\omO}(E) \ |\ \omO \in \{0,1\}^N \}=\spN.
\end{equation}
By applying Theorem \ref{thm_BG03}, we obtain that
\begin{equation}\label{eq_prop1_11}
\forall \ell \in (0,\ell_C),\ \forall E\in I(N,V,\ell)\setminus \SV, G_{\{0,1\}}(E)\ \mathrm{is}\ \mathrm{dense}\ \mathrm{in}\ \SpN.
\end{equation}
Now, as the F\"urstenberg group $G(E)$ is the closure of $G_{\{0,1\}}(E)$, we get :
\begin{equation}\label{eq_prop1_12}
\forall \ell \in (0,\ell_C),\ \forall E\in I(N,V,\ell)\setminus \SV, G(E)=\SpN.
\end{equation}
We have proved Proposition \ref{prop_GE_SpN} because $\mathcal{V}$ is of Lebesgue measure $0$ (see (\ref{eq_V_Leb_zero})) and $\SV$ is finite.
\end{proof}

\noindent We deduce Corollary \ref{cor_lyap_exp} by using the fact that, for $\ell$ and $E$ such that $G(E)=\SpN$, $G(E)$ is \LpSI\ (see \cite[Definitions A.IV.3.3 and A.IV.1.1]{BL85} for the definitions of these notions). Thus, by \cite[Proposition IV.3.4]{BL85}, we get the separability and the positivity of the Lyapunov exponents $\gamma_1(E),\ldots,\gamma_N(E)$ (see (\ref{eq_cor_lyap_exp})). Because $\SV$ is finite, it is of Lebesgue measure zero in $\R$ and we can apply Kotani's theory (see \cite{KS88}) to prove the absence of absolutely continuous spectrum in $I(N,V,\ell)$, for $\ell \in (0,\ell_C)$ and $V\in \SN \setminus \mathcal{V}$, which finish to prove Corollary \ref{cor_lyap_exp}.

\begin{rem}\label{rem_holder_lyap}
We also note that, by applying \cite[Theorem 2]{boumaza_rmp}, we get that the functions $E\mapsto \gamma_p(E)$ for $p\in \aN$ are H\"older continuous on every compact interval $I\subset I(N,V,\ell)\setminus \SV$, for $\ell \in (0,\ell_C)$ and $V\in \SN \setminus \mathcal{V}$.
\end{rem}

\section{Proof of Theorem \ref{thm_localization}}\label{sec_proof_loc}

\noindent Using Proposition \ref{prop_GE_SpN}, Theorem \ref{thm_localization} will be a consequence of the following result.

\begin{thm}[Theorem 1, \cite{boumaza_mpag2}]\label{thm_loc_mpag2}
Let $I\subset \R$ be a compact interval such that $\Sigma\cap I\neq \emptyset$ and let $\tilde{I}$ be an open interval, $I\subset \tilde{I}$, such that, for every $E\in \tilde{I}$, $G(E)$ is \LpSI. Then, $\Hl$ exhibits (EL) and (SDL) in $I$.
\end{thm}

To prove this result we had to :

\begin{enumerate}[1.]
\item Obtain an integral representation of the Lyapunov exponents of $\Hl$ which, in particular, implies their positivity.
\item Deduce from this integral representation some H\"older regularity of the Lyapunov exponents (see Remark \ref{rem_holder_lyap}).
\item Show that the integrated density of states of $\Hl$ has the same H\"older regularity (see Proposition \ref{prop_holder_ids}).
\item Prove a Wegner estimate using the H\"older regularity of the integrated density of states.
\item Obtain (EL) and (SDL) by using multiscale analysis.
\end{enumerate}

\begin{proof}[Proof of Theorem \ref{thm_localization}]
Let $V\in \SN \setminus \SV$ and assume that $\ell \in (0,\ell_C)$. Let $\tilde{I} \subset  I(N,V,\ell)\setminus \SV$ be an open interval such that there exists $I\subset \tilde{I}$, a compact interval with $\Sigma\cap I\neq \emptyset$. If we take $\ell$ small enough, as the intervals $I(N,V,\ell)$ tends to $\R$, we can always find such intervals $\tilde{I}$ and $I$. Now, as $\tilde{I} \subset  I(N,V,\ell)\setminus \SV$, by Proposition \ref{prop_GE_SpN}, for every $E\in \tilde{I}$, $G(E)=\SpN$. Thus, we can apply Theorem \ref{thm_loc_mpag2} to obtain that $\Hl$ exhibits (EL) and (SDL) in $I$, which proves Theorem \ref{thm_localization}.
\end{proof}

\section{Results on the integrated density of states}\label{sec_ids}

The integrated density of states is the distribution function of the energy levels of $\Hl$, per unit volume. To define it properly, we first need to restrict the operator $\Hl$ to finite length intervals. Let $L\geq 1$ be an integer and let $H_{\ell}^{(L)}(\omega)$ be the restriction of $\Hl$ to $L^2([-\ell L,\ell L])\otimes \R^N$, with Dirichlet (or Neumann) boundary conditions at $\pm \ell L$.

\begin{defi}\label{def_ids}
The integrated density of states associated to $\Hl$ is the function from $\R$ to $\R_{+}$, $E\mapsto N(E)$, where $N(E)$, for $E\in \R$, is defined as : 
\begin{equation}\label{eq_def_ids}
N(E)=\lim_{L\to +\infty} \frac{1}{2\ell L} \# \{ \lambda \leq E |\ \lambda \in \sigma(H_{\ell}^{(L)}(\omega)) \},
\end{equation}
for $\mathsf{P}$-almost every $\omega \in \Omega$.
\end{defi}
\vskip 2mm

\noindent For the integrated density of states associated to $\Hl$, we have the following results.

\begin{prop}\label{prop_holder_ids}
$(1)$ For any $V\in \SN$, $\ell >0$ and $E\in \R$, the limit (\ref{eq_def_ids}) exists and is $\mathsf{P}$-almost surely independent of $\omega \in \Omega$.

\noindent $(2)$ Let $V\in \SN \setminus \mathcal{V}$ and $\ell \in (0,\ell_C)$. Let $I \subset  I(N,V,\ell)\setminus \SV$ be an open interval. Then the integrated density of states of $\Hl$, $E\mapsto N(E)$, is H\"older continuous on $I$.
\end{prop}

\begin{proof}
For point $(1)$, we directly apply  Corollary $1$ in \cite{boumaza_rmp}. For point $(2)$, we use the fact that, for $E\in I$, with $I \subset I(N,V,\ell)\setminus \SV$, we have $G(E)=\SpN$ and thus, $G(E)$ is \LpSI. As $I(N,V,\ell)$ is compact, we can directly apply to $I$ the Theorem 4 of \cite{boumaza_rmp} which proves that $\Hl$ is H\"older continuous on $I$.
\end{proof}




\begin{thebibliography}{99}

\bibitem{BL85} 
P. Bougerol and J. Lacroix,  
\emph{Products of Random Matrices with Applications to Schr\"odinger
Operators}, Progr. Probab. Statist. \textbf{8}, Birkh\"auser, Boston,
(1985)

\bibitem{boumaza_rmp}
H. Boumaza,
\emph{H\"older continuity of the integrated density of states for matrix-valued Anderson models}, Rev. Math. Phys. \textbf{20}(7), 873\textendash 900 (2008), DOI:10.1142/S0129055X08003456

\bibitem{boumaza_mpag2}
H. Boumaza,
\emph{Localization for a matrix-valued Anderson model}, Math. Phys. Anal. Geom. \textbf{12}(3), 255\textendash 286 (2009), DOI:10.1007/s11040-009-9061-3

\bibitem{BK05}
J. Bourgain and C.E. Kenig,
\emph{On localization in the continuous Anderson-Bernoulli model in higher dimension}, Invent. Math.  \textbf{161}(2) 389\textendash 426 (2005)

\bibitem{BG03}
E. Breuillard and T. Gelander, 
\emph{On dense free subgroups of Lie groups}, 
J.\ Algebra \textbf{261}(2), 448\textendash 467 (2003)

\bibitem{DSS02} 
D. Damanik and R. Sims and G. Stolz,  
\emph{Localization for one-dimensional, continuum, Bernoulli-Anderson models}, 
Duke Mathematical Journal \textbf{114}, 59\textendash 99 (2002)

\bibitem{KS88} 
S. Kotani and B. Simon,   
\emph{Stochastic Schr\"odinger operators and Jacobi Matrices on the Strip}, 
Comm. Math. Phys. \textbf{119}(3), 403\textendash 429 (1988)


\end{thebibliography}
\end{document}